\documentclass[12pt]{article}
\usepackage{a4wide}

\usepackage[a4paper, margin=2.3cm]{geometry}
\usepackage{amsmath,amssymb,amsthm}
\usepackage{color}
\usepackage{parskip}
\usepackage{parskip}
\usepackage{setspace}
\usepackage{graphicx}
\usepackage{cases}
 
\newtheorem{theorem}{Theorem}[section]

\newtheorem{lemma}[theorem]{Lemma}
\newtheorem{definition}[theorem]{Definition}

\newtheorem*{definition*}{Definition}

\def\R{\mathbb{R}}

\begin{document}
\title{On three-variable expanders over finite valuation rings}

\author{
 Nguyen Van The \thanks{VNU University of Science, Vietnam National University, Hanoi. Email: nguyenvanthe\_t61@hus.edu.vn} \and Phuc D Tran \thanks{Department of Mathematics \& Sciences, American University in Bulgaria. Email: pnt170@aubg.edu} \and Le Quang Ham \thanks{VNU University of Science, Vietnam National University, Hanoi. Email: hamlaoshi@gmail.com} \and Le Anh Vinh\thanks{Vietnam Institute of Educational Sciences. Email: vinhle@vnies.edu.vn. }}
\date{}
\maketitle  

\begin{abstract}
Let $\mathcal{R}$ be a finite valuation ring of order $q^r$. 
In this paper, we prove that for any quadratic polynomial $f(x,y,z) \in \mathcal{R}[x,y,z]$ that is of the form $axy+R(x)+S(y)+T(z)$ for some one-variable polynomials $R, S , T$, we have 
\[ \left|f(A,B,C)\right| \gg \min\left\{ q^r, \frac{|A||B||C|}{q^{2r-1}}\right\}\]
for any $A, B, C \subset \mathcal{R}$.
We also study the sum-product type problems over finite valuation ring $\mathcal{R}.$ More precisely, we show that for any $A \subset \mathcal{R}$ with $|A| \gg q^{r-1/3}$ then
$$\max\{ |A \cdot A|, |A^d + A^d|\},\max\{ |A + A|, |A^2 + A^2|\},\max\{|A-A|,|AA+AA|\} \gg |A|^{2/3}q^{r/3},$$
and  $|f(A) + A| \gg |A|^{2/3}q^{r/3}$ for any one variable quadratic polynomial $f$.
\end{abstract}

\section{Introduction}

Let $\mathbb{F}_q$ be a finite field of order $q$ where $q$ is an odd prime power. Given a function $f \colon \mathbb{F}_q^d\to \mathbb{F}_q$, the set
\[f(A, \ldots, A)=\{f(a_1, \ldots, a_d)\colon a_1, \ldots, a_d\in A\},\]
is called the image of the set $A^d\subset \mathbb{F}_q^d$ under the function $f$. We will use the following definition of expander polynomials, which can be found in \cite{hls}. 
\bigskip
\begin{definition}
Let $f$ be a function from $\mathbb{F}_q^d$ to $\mathbb{F}_q$. 
\begin{itemize}
\item[1.] The function $f$ is called a strong expander with the exponent $\varepsilon>0$ if for all $A\subset \mathbb{F}_q$ with $|A|\gg q^{1-\varepsilon}$, one has $|f(A, \ldots, A)|\ge q-k$ for some fixed constant $k$. 
\item[2.] The function $f$ is called a moderate expander with the exponent $\varepsilon>0$ if for all $A\subset \mathbb{F}_q$ with $|A|\gg q^{1-\varepsilon}$, one has $|f(A, \ldots, A)|\gg q$.
\end{itemize}
\end{definition}
Here and throughout, $X \ll Y$ means that there exists  some absolute constant $C_1>0$ such that $X \leq C_1Y$, $X \gtrsim Y$ means $X\gg (\log Y)^{-C_2} Y$ for some absolute constant $C_2>0$, and  $X\sim Y$ means $Y\ll X\ll Y$.

The study of expander polynomials over finite fields has been of much research interest in recent years. For two variable expanders, Tao \cite{tao} gave a general result that for any polynomial $P(x, y)\in \mathbb{F}_q[x, y]$ that is not one of the forms $Q(F_1(x)+F_2(y))$ and $Q(F_1(x)F_2(y))$ for some one variable polynomials $Q, F_1, F_2$, we have 
\[|P(A, A)|\gg q,\]
under the assumption $|A|\gg q^{1-\frac{1}{16}}$. This implies that such polynomials $P(x,y)$ are moderate expanders with the exponent $\varepsilon=1/16$. 

For three variable expanders over prime fields, Pham, de Zeeuw, and the fourth listed author \cite{pham} showed that any quadratic polynomial in three variables $P\in \mathbb{F}_p[x,y,z]$ that is not independent on any variable and that does not have the form $G(H(x)+K(y)+L(z))$ for some one variable polynomials $G, H, K, L$ is a moderate expander with the exponent $\epsilon=1/3$. Note that, the exponent $1/3$ can be improved in the case of rational function expanders. More precisely, Rudnev, Shkredov, and Stevens \cite{RSS} showed that the function $(xy-z)(x-t)^{-1}$ is a moderate expander with the exponent $\epsilon=17/42$ over prime fields.

Over general finite fields, for three and four variable expanders, there are several known families of moderate expanders with the exponent $\varepsilon = 1/3$ constructed by various authors. For example, the expander $(x - y)(z-t)$ is constructed by Bennett, Hart, Iosevich, Pakianathan, and Rudnev \cite{bennett}, the expander $xy+zt$ is constructed by Hart and Iosevich \cite{ha2}, the expander $x+yz$ is constructed by Shparlinsk \cite{shpas}, and the expanders $x(y+z)$ and $x+(y-z)^2$ are constructed by the fourth listed author \cite{vinh2}. Using methods from spectral graph theory, one can break the exponent $1/3$ by showing that $(x-y)^2+zt$ is a moderate expander with $\varepsilon=3/8$ \cite{vinh2}. Until now, there is no constructed moderate expander with a better exponent.

Let $\mathcal{R}$ be a finite valuation ring of order $q^r.$ Throughout, $\mathcal{R}$ is assumed to be commutative, and to have an identity. Let us denote the set of units, non-units in $\mathcal{R}$ by $\mathcal{R}^*, \mathcal{R}^0,$ respectively. 

In the setting of finite valuation ring, the study of expanding polynomials  have been of great interest and was considered in many different context by various authors in the literature. In \cite{ham}, Pham, the first and the fourth listed authors concluded a result for two-variable polynomials. More precise, the third and fourth authors showed that for any set $A \subset \mathcal{R}\setminus \left\{ \mathcal{R}^0, \mathcal{R}^0 - 1\right\}$ then 
\[ |A(A+1)| \gg \min\left\{ \sqrt{q^r|A|}, \frac{|A|^2}{\sqrt{q^{2r-1}}}\right\}.\] 

For three-variable polynomials, Yazici \cite{EAY} developed a point-plane incidence estimate over finite valuation ring and showed that the function $xy+z$ is a moderate expander over $\mathcal{R}$ with the exponent $1/3$. More precisely, she proved the following result.   

\begin{theorem}[\text{\cite[Theorem 1.1]{EAY}}]\label{cited.theo1}
 Let $\mathcal{R}$ be a finite valuation ring of order $q^r$ and $A,B,C$ be the subsets of $\mathcal{R}$, then 
\begin{equation*}
|AB+C| \gg \min \left\{ q^r, \frac{|A||B||C|}{q^{2r-1}} \right\}. 
\end{equation*}
\end{theorem}
In \cite{anh}, Anh and her colleagues proved the following result. 
\begin{theorem}[\text{\cite[Lemma 4.4]{anh}}]\label{Anh.Theo}
Let $X, Y, Z$ be sets in $\mathcal{R}.$ We have
\[ \left| (X-Y)^2+Z \right| \gg \min\left\{ q^r, \dfrac{|X||Y||Z|}{q^{2r-1}}\right\}.\]
\end{theorem}
As a direct consequence, the polynomial $(x-y)^2+z$ is a moderate expander with the exponent $1/3$. They also provided a family of four variable moderate expanding polynomials over finite valuation rings with the exponent $3/8$. More precisely, they showed that $x(x + t)y + z, x(x + t) + yz, x(x + t)(y + z), y (x(x + t) + z), (x(x + t)- y)^2 + z$ and $(y - z)^2 + x(x + t)$ are moderate expanders over $\mathcal{R}$ with the exponents $3/8.$ 

For more background about this problem and an extensive exploration of the subject over finite valuation rings, we refer the reader \cite{Bor, Erd, Mur, rud2, tao, Thang} and the references therein. 

In this paper, following the spirit of Yazici's point-plane incidence, we will present several results for the large sets. Our first result is the following.
\begin{theorem} \label{mainth3-1}
Let $\mathcal{R}$ be a finite valuation ring of order $q^{r}$. Consider a polynomial
	$$f(x,y,z)=axy + R(x)+S(y)+T(z),$$
here $R, S, T \in \mathcal{R}[u]$ are polynomials of degree at most two, $a \neq 0$, and $T(z)$ is not a constant with the highest coefficient $m \in \mathcal{R}^*.$ Suppose that $A,B, C\subseteq \mathcal{R}$ and we further assume that $|C|\geq 2q^{r-1}$ if $T$ has degree two. We have the following estimation
	\begin{align*}
	|f(A,B,C)| \geq \frac{1}{8} \min \left\{ q^{r}, \frac{|A||B||C|}{q^{2r-1}} \right\}.
	\end{align*}
\end{theorem}

Note that, Theorem  \ref{cited.theo1} and Theorem \ref{Anh.Theo} are special cases of Theorem \ref{mainth3-1}.

In \cite{EAY}, Yazici obtained the following sum-product type estimate over an arbitrary finite valuation ring $\mathcal{R}.$

\begin{theorem}[\text{\cite[Theorem 1.2]{EAY}}]\label{cited.theo2}For any subset $A$ of $\mathcal{R}$ with $|A+A||A|^2 >q^{3r-1}$, then 
\begin{equation*}
|A^2+A^2||A+A| \gg q^{\frac{r}{2}}|A|^{\frac{3}{2}}.
\end{equation*}
\end{theorem}

Using the spectral graph method, Hiep \cite{pham1} proved the following improvement of Theorem \ref{cited.theo2}. 
\begin{theorem}[\text{\cite[Theorem 1.5]{pham1}}] \label{coro.A^2+A^2}
Let $A$ be a subset of $\mathcal{R}.$ Suppose that $|A| \geq 2q^{r-1}$ and $|A+A||A|^2 \ge q^{3r-1}$, then we have
\[
|A^2 + A^2||A+A|^2 \geq \dfrac{1}{2}|A|^{2}q^{r}.\]
This implies that           
\[ \max\left\{ |A+A|, |A^2+ A^2|\right\} \geq 2^{-\frac{1}{3}}\,|A|^{2/3}q^{r/3}.\]
\end{theorem}
In this paper, we give another proof of Theorem \ref{coro.A^2+A^2} that is closely related to the proof of Theorem \ref{mainth3-1}. By a slightly different technique, we also give a lower bound in term $A^3 + A^3$ instead of $A^2+A^2$ in Theorem \ref{coro.A^2+A^2}.
\begin{theorem}\label{theo.A^3+A^3}
For $A \subset \mathcal{R}$ with $\dfrac{|A+A|^{4}}{|A|} \ge q^{3r-1},$ then we have 
\[ \max\left\{ |A+A|, |A^3+A^3| \right\} \gg q^{r/10} |A|^{9/10}.\]  
\end{theorem} 

The idea of Theorem \ref{mainth3-1} also leads to the following results.

\begin{theorem} \label{theo.f(A)+A}
Let $f$ be a one-variable quadratic polynomial with coefficients over $\mathcal{R}$. If $|A+f(A)||A|^{2} \ge q^{3r-1}$, then
\begin{align*}
|f(A)+A| \geq 2^{-\frac{1}{3}}\,|A|^{2/3}q^{r/3}.
\end{align*}
\end{theorem}

\begin{theorem}\label{AA+AA}
For $A \subset \mathcal{R}$ with $|A| \geq q^{r-1/3},$ then we have
\[\max\left\{ |A-A|, |AA + AA|\right\}\ge 2^{-\frac{1}{3}} |A|^{2/3}q^{r/3}. \]
\end{theorem}

We also develop a weighted version of Yazici's incidence to prove the sum-product estimates over finite valuation ring. More precisely, we have the following theorems. 
\begin{theorem} \label{AA.A^d+A^d}
Let $d$ be an integer with $d \ge 1$ and $A$ be a subset of $\mathcal{R}.$ Suppose that  $|AA||A|^{2} \geq q^{3r-1}$, then we have
\[ |A^{d}+A^{d}||AA|^{2} \gg q^{r}|A|^{2},\]
or
\[ \max \left\{ |A^{d}+A^{d}|, |AA| \right\} \gg q^{\frac{r}{3}}|A|^{2/3}.\] 
\end{theorem}

Note that, the condition $|A| \gg q^{r-\frac{1}{3}}$ implies the conditions in previous theorems. Therefore, we can replace all conditions in previous theorem by the condition $|A| \gg q^{r-\frac{1}{3}}$

\section{Preliminaries}
We start this section by recalling the definition of finite valuation rings. 
\begin{definition}
	Finite valuation rings are finite rings that are local and principal.
\end{definition}
Throughout, rings are assumed to be commutative, and to have an identity. Let $\mathcal{R}$ be a finite valuation ring, then $\mathcal{R}$ has a unique maximal ideal that contains every proper ideals of $\mathcal{R}$. This implies that there exists a non-unit $z$ called \textit{uniformizer} in $\R$ such that the maximal ideal is generated by $z$. Moreover, we also note that the uniformizer $z$ is defined up to a unit of $\mathcal{R}$. 

There are two structural parameters associated to $\mathcal{R}$ as follows: the cardinality of the residue field $F=\mathcal{R}/(z)$, and the nilpotency degree of $z$, where the nilpotency degree of $z$ is the smallest integer $r$ such that $z^r=0$.  Let us denote the cardinality of $F$ by $q$. 
In this note, $q$ is assumed to be odd, then $2$ is a unit in $\mathcal{R}$.

If $\mathcal{R}$ is a finite valuation ring, and $r$ is the nilpotency degree of $z$, then we have a natural valuation
\[\nu\colon \mathcal{R}\to \{0,1,\ldots,r\}\]
defined as follows: $\nu(0)=r$, for $x\ne 0$, $\nu(x)=k$ if $x\in (z^k)\setminus (z^{k+1})$. We also note that $\nu(x)=k$ if and only if $x=uz^k$ for some unit $u$ in $\mathcal{R}$. Each abelian group $(z^k)/(z^{k+1})$ is a one-dimensional linear space over the residue field $F=\mathcal{R}/(z)$, thus its size is $q$. This implies that $|(z^k)|=q^{r-k}, ~k=0,1,\ldots,r$. In particular, $|(z)|=q^{r-1}, |\mathcal{R}|=q^r$ and $|\mathcal{R}^{*}|=|\mathcal{R}|-|(z)|=q^r-q^{r-1}$, (for more details about valuation rings, see \cite{ati,bi,fulton,bogan}). The following are some examples of finite valuation rings:
\begin{enumerate}
	\item Finite fields $\mathbb{F}_q$, $q=p^n$ for some $n>0$.
	\item Finite rings $\mathbb{Z}_{p^r}$, where $p$ is a prime.
	\item $\mathcal{O}/(p^r)$ where $\mathcal{O}$ is the ring of integers in a number field and $p\in \mathcal{O}$ is a prime.
	\item $\mathbb{F}_q[x]/(f^r)$, where $f\in \mathbb{F}_q[x]$ is an irreducible polynomial.
\end{enumerate}

For the proofs of Theorem \ref{mainth3-1}, Theorem \ref{theo.f(A)+A}, and Theorem \ref{AA+AA}, we use the following incidence theorem between points and planes in $\mathcal{R}^3,$ which given by Yazici \cite{EAY}.

\begin{theorem}[\text{\cite{EAY}}] \label{P-P-Y}
  Let $\mathcal{R}$ be a finite valuation ring of order $q^r$. Let $Q$ be a set of point in $\mathcal{R}^3$ and $\Pi$ be a set of planes in $\mathcal{R}^{3}$. Then the number of incidences $I(Q,\Pi)$ satisfies:
\begin{align*}
\left| I(Q,\Pi)-\frac{1}{q^{r-1}}\frac{q^2+q+1}{q^3+q^2+q+1}\,|Q||\Pi|\right| \leq  q^{2r-1}\,|Q|^{1/2}|\Pi|^{1/2}. 
\end{align*}                   
Hence,
\begin{align*}
|I(Q,\Pi)| \leq \frac{1}{q^r}\,|Q||\Pi|+q^{2r-1}\,|Q|^{1/2}|\Pi|^{1/2}. 
\end{align*}
\end{theorem}

To prove Theorem \ref{AA.A^d+A^d}, we will need the following weitghted version of the point-plane incidences over finite valuation rings.

\begin{theorem}[\textsl{Point-plane incidences - weighted version}] \label{CorY1}
Let $Q, \Pi$ be weighted set of points and planes in $\mathcal{R}^3$ with the weighted integer function $\omega,$ both total weight $W.$ 
Suppose that maximum weights are bounded by $\omega_0 \geq 1.$ Define the number of weighted incidences is 
\begin{align*}
I_{\omega}= \sum_{q \in Q, \, \pi \in \Pi} \omega(q)\omega(\pi)\,\delta_{q\,\pi},
\end{align*}
where
\begin{subnumcases}{\delta_{q \pi}:=}
1 \quad \text{if}\quad  q \in \pi, \notag\\
0 \quad \text{if} \quad q \notin \pi.\notag
\end{subnumcases}
Then the number $I_{\omega}$ of weighted incidences is bounded as follows
\begin{align*}
I_{\omega} = \displaystyle \sum_{ q \in \pi} \omega(q) \omega(\pi) \ll \frac{1}{q^{r}}\,W^{2}+q^{2r-1}W.
\end{align*}
\end{theorem}
 
\begin{proof}
In this proof, we will use a weight rearrangement argument. Firstly, we rewrite the number of weighted incidences as $$I_{\omega}=\displaystyle \sum_{q \in Q} \omega (q)\,\left(\sum \omega(\pi)\right)=\displaystyle \sum_{q \in Q} \omega(q)\,W(q),$$
where 
\[ W(q)=\sum_{\pi \in \Pi:\, q \in \pi} \omega(\pi).\]

Since no weight is larger than $\omega_0$, we always can pick a subset $Q' \subset Q$, containing $n = \left\lceil\frac{W}{w_0} \right\rceil$ ``richest'' points $q_1,q_2,...,q_n$ in terms of $W(q)$. We replace $Q$ by $Q'$ by assign to each one of the points in $Q'$ the weight $\omega_0$ and deleting the rest of points in $Q$. Without loss of generality, we can assume that $W(q_n)$ is minimum of $W(q_k)$ for $1 \leq k \leq n$. Note that 
\begin{align*}
I_{\omega}(Q,\Pi)&=  \sum_{q \in Q'} \omega (q) W(q) + \sum_{q \in Q \setminus Q'} \omega(q)\,W(q) \\
&\leq \sum_{q \in Q'} \omega(q) W(q) + \left(W-\sum_{q \in Q'} \omega(q)\right)\, W(q_n)\\
&\leq \sum_{q \in Q'} \omega(q) \left(W(q) - W(q_n)\right) + n \omega_0 W(q_n)  \\
&\leq \sum_{k=1}^{n} \omega_0 W(q_i).
\end{align*}

The number of weighted incidences $I_{\omega}$ will thereby not decrease. Similarly, we can pick a subset $\Pi'$ of the plane set $\Pi$, containing $n$ richest planes in terms of their non-weighted incidences with $Q'$. We also assign the weight $\omega_0$ to each plane in $\Pi'$. We now replace $Q, \Pi$ by $Q',\Pi'$. 

Applying the Theorem \ref{P-P-Y}, we obtain 
\begin{align*}
I_{\omega}(Q,\Pi) &\leq I_{\omega_0,\omega}(Q',\Pi)) \leq I_{\omega_0}(Q',\Pi') \\
&\leq  \frac{1}{q^r} w_0^2 n^2 + q^{2r-1} w_0 n \\
&\ll \frac{1}{q^r}W^2+q^{2r-1}W,
\end{align*}
which completes the proof of Theorem \ref{CorY1}.
\end{proof} 

\section{Proof of Theorem \ref{mainth3-1}}

Let $N$ be the number of tuples $(x,y,z,x',y',z') \in (A\times B \times C)^2$ such that $f(x,y,z)=f(x',y',z')$. We rewrite this equation as
\[ (ay) x-(ax')y'+\left( R(x) -S(y') + T(z) \right) = R(x') - S(y) + T(z'). \]
We define the point set $Q$ and the plane set $\Pi$ as following
\begin{align*}
Q &:=\left\{ \left(x,\, y', \,R(x) -S(y') + T(z)\right)\,: \, (x,y',z) \in A \times B \times C \right\}, \\
\Pi &:= \left\{ (ay) \cdot X - (ax') \cdot Y + Z = R(x') - S(y) + T(z') \,:\, (x',y,z') \in A \times B \times C \right\}.
\end{align*}

For each choice of $x,y'$, and $n$, we count the number of solutions of the equation $T(z) = n$. If deg$(T) = 1$, it is clear that the equation has exactly one solution since its highest coefficient $m$ is a unit of $\mathcal{R}.$ If deg$(T) = 2$, we rewrite the equation as follows
	\begin{align}\label{eq1}
		(z-u)^2=n',
	\end{align}
where $u$ is a fixed number and does not depend on $x,y',n.$ 
If there does not exist $z_1$ such that $z^2_1=n'$ then \eqref{eq1} has no solution. Otherwise, we have
	\begin{align}\label{eq2}
		&(z-u-z_1)(z-u+z_1)=0\\
		&z-u-z_1=u_1 \mu^{\alpha_1}\quad \mbox{and}\quad z-u+z_1=u_2 \mu^{\alpha_2}
	\end{align}
where $\mu$ is \textit{uniformizer} in $\mathcal{R}.$
	It follows that $z-u=2^{-1}(u_1\mu^{\alpha_1}+u_2\mu^{\alpha_2})$ is non unit.
	Since $|C|\geq 2q^{r-1},$ without loss of generality, we can assume that $z-u$ is unit so the equation \eqref{eq2} has at most two solutions. This implies that $T(z) = n$	has at most two solutions. Therefore, each point $(u,v,w) \in Q$ corresponds to at most two points $(x,y',z) \in A \times B \times C $. 
	
	Apply the same argument to the set of planes $\Pi$, we have 
\[|Q|,|\Pi| \leq |A||B||C|,\]
and 
\[ N \leq  4 \cdot I(Q,\Pi),\]
where $I(Q,\Pi)$ is number of incidences between $Q$ and $\Pi$. Now applying the Cauchy-Schwarz inequality and Theorem \ref{P-P-Y}, we get
\begin{align*}
\left(|A||B||C|\right)^{2} &\leq |f(A, B, CC)| \cdot N \\
 &\leq 4 \left| f(A , B, C)\right| \cdot I(Q,\Pi) \\
 &\leq 4 \left| f(A, B, C)\right|\left( \frac{1}{q^r}|Q||\Pi|+q^{2r-1}|Q|^{1/2}|\Pi|^{1/2}\right) \\
 &\leq 4 \left|f(A, B, C)\right|\left(\dfrac{|A|^2|B|^2|C|^2}{q^{r}}+q^{2r-1}|A||B||C| \right).
\end{align*}

This implies that
\[ \left|f(A,B,C)\right| \geq \frac{1}{8} \min \left\{ q^{r}, \frac{|A||B||C|}{q^{2r-1}} \right\},\]
completing the proof of Theorem \ref{mainth3-1}.

\section{Proofs of Theorem \ref{coro.A^2+A^2} and Theorem \ref{theo.A^3+A^3}} 

\begin{proof}[Proof of Theorem \ref{coro.A^2+A^2}.]
Let $f(x,y,z) = (x-z)^2+y^2$, we consider the following equation 
\begin{align} \label{equ2}
f(x,y,z)= t,
\end{align}
where $x \in A+A, \, y,z \in A, \, t \in A^2+A^2.$ For any triple $(u,v,w) \in A$, a solution of the equation \eqref{equ2} is given by $x = u + w \in A + A, \,  y = v \in A,\, z = u \in A,\,  t=w^2+v^2 \in A^2+A^2$. Thus, there are at least $|A|^3$ solutions of the equation \eqref{equ2}. Now applying the Cauchy-Schwarz inequality, we get
\begin{equation}\label{equa3A^2+A^2}
|A|^{6} \leq \left| A^2+A^2\right| \cdot \left|\left\{ (x,y,z,x',y',z') \in \left((A + A)\times A \times A \right)^{2} \,:\, f(x,y,z)=f(x',y',z')\right\}\right|
\end{equation}
We proceed similarly as in the proof of Theorem \ref{mainth3-1} for the specific function 
\[f(x,y,z)=-2xz+x^2+z^2+y^2\]
with $|A|=|B|, |C|=|A+A|.$ It follows from \eqref{equa3A^2+A^2} that
\begin{align*}
|A|^{6} \leq \left|A^2+A^2\right| \cdot \left(\dfrac{|A+A|^{2}|A|^{4}}{q^{r}}+q^{2r-1}\,|A+A||A|^{2}\right). 
\end{align*}
Hence,
\[|A|^2q^r \leq \left|A^2+A^2 \right||A+A|^2\left( 1+\dfrac{q^{3r-1}}{|A+A||A|^2} \right).\]

Thus, under the our assumption $|A+A|.|A|^{2}\geq q^{3r-1},$ we have 
\[|A+A|^{2} |A^2+A^2| \geq \dfrac{1}{2}|A|^{2}q^{r},\]
which completes the proof of Theorem \ref{coro.A^2+A^2}. 
\end{proof}

For the proof of Theorem \ref{theo.A^3+A^3}, we will need the following Plunnecke-Ruzsa inequality. 

\begin{lemma}
Let $A$ and $B$ be finite subsets of an abelian group such that $|A + B| \leq K|A|.$ Then for an arbitrary $0 < \delta < 1$, there is a nonempty set $X \subset A$ such that $|X| \geq (1 - \delta)|A|$, and for any integer $k,$ one has
 \begin{equation}\label{Plun.ine}
 	|X + kB| < \left( \dfrac{K}{\delta}\right)^k |X|.
 \end{equation}   
\end{lemma}

\begin{proof}[\bf Proof of Theorem \ref{theo.A^3+A^3}]
Let $b=d-a$, where $a,b \in A, d \in A+A$. We have
\begin{align*}
a^3+(d-a)^3 &= d\left[a^2-a(d-a)+(d-a)^2 \right] \\
&= d (3a^2-3ad+d^2)\\
&= 3d\left[ (a-d/2)^2+d^2/12 \right]\\
&= 3d(s^2+d^2/12),
\end{align*}
where $s=a-d/2.$ 

Let $E$ be the number of solutions of the following equation 
\begin{equation} \label{equa4}
c^3+b^3+a^3=c'^3+b'^3+a'^3 : (a,b,c,a',b',c') \in A^6.
\end{equation}
By Plunnecke-Ruzsa inequality, we can choose a large subset $A' \subset A$ such that
\begin{equation} \label{iequ4}
|A'^3+A'^3+A'^3|\ll \frac{|A^3+A^3|^2}{|A|}.
\end{equation} 
Putting \eqref{equa4} and \eqref{iequ4} together with the Cauchy-Schwarz inequality,  we have
\begin{equation} \label{Ieu}
\frac{|A^3+A^3|^2}{|A|} \gg \frac{|A|^6}{E}.
\end{equation}
Besides, we can rewrite the equation \eqref{equa4} as $3ds^2+d^3/4+c^3 = 3d's'^2+d'^3/4+c'^3$. This is equivalent to
\begin{align*}
 s^2\cdot(3d) - 3d's'^2-d'^3/4-c'^3 = -d^3/4-c^3.
\end{align*}
Define the point set $Q$ and the plane set $\Pi$ as follows
\begin{align*}
Q &:= \left\{ (s^2,3d',-d'^3/4-c'^3): \, s \in A - \frac{A+A}{2},d' \in A+A, c' \in A \right\}, \\
\Pi &:= \left\{ 3d \cdot X - s'^2 \cdot Y + Z = -d^3/4-c^3:\, s' \in A-\frac{A+A}{2}, d \in A+A, c \in A \right\}.
\end{align*}
For each $(u,v,w)\in Q$, there are at most six triples $(s,d',c') \in \left(A-\dfrac{A+A}{A}\right) \times A+A \times A$ such that $u=s^2,v=3d', w=-d'^3/4-c'^3.$ 

Apply the same argument to the plane set $\Pi$, we have
\begin{equation} \label{equa5}
|Q|,|\Pi| \leq \left|A-\dfrac{A+A}{A}\right|\, |A+A| \,|A|,
\end{equation}
and 
\[ E \leq 36.I(Q,\Pi),\]
where $I(Q,\Pi)$ is the number of incidences between $Q$ and $\Pi$.
Applying the Theorem \ref{P-P-Y}, we obtain 
\begin{equation} \label{iequa7}
E \leq 36 \left( \frac{1}{q^r} |Q||\Pi| +q^{2r-1}|Q|^{1/2}|\Pi|^{1/2} \right).
\end{equation}
Now, we bound the number of values of $s$. Note that 
\begin{equation} \label{iequa5}
\left|A -\frac{A+A}{2}\right| = |2A - A - A| \leq |A+A-A-A| \leq \frac{|A+A|^3}{|A|^2}.
\end{equation}
The last inequality is obtained by applying the Plunnecke-Ruzsa inequality twice.

From \eqref{equa5} and \eqref{iequa5}, we have
\begin{equation} \label{iequ6}
|Q|,|\Pi| \leq \frac{|A+A|^4}{|A|}.
\end{equation}
From \eqref{iequa7} and \eqref{iequ6}, we have
\begin{equation} \label{iequ8}
E \leq 36 \left( \frac{1}{q^r} \frac{|A+A|^8}{|A|^2}+q^{2r-1}.\frac{|A+A|^4}{|A|} \right).
\end{equation}
Putting \eqref{Ieu} and \eqref{iequ8} together, we have 
\[ |A|^7 \ll |A^3+A^3|^2 \cdot E \ll |A^3+A^3|^2 \left( \frac{1}{q^r} \frac{|A+A|^8}{|A|^2}+q^{2r-1} \cdot \frac{|A+A|^4}{|A|} \right). \]
Hence,
\[ q^r|A|^9 \ll \left|A^3+A^3\right|^2|A+A|^8 \left( 1+\frac{q^{3r-1}\,|A|}{|A+A|^4} \right).\]
Therefore, if $q^{3r-1} \ll \frac{|A+A|^4}{|A|}$ then
\[ \max\left\{|A+A|,|A^3+A^3|\right\} \gg q^{r/10}|A|^{9/10},\]
which completes the proof of Theorem \ref{theo.A^3+A^3}.
\end{proof}

\subsection{Proof of Theorem \ref{theo.f(A)+A}}
Without loss of generality, we can assume that $f(x)=ax^{2}+bx$ with $a \neq 0$. We consider the following equation
\begin{equation}\label{equa1}
a(x-y)^{2} + b(x-y) + z = t,
\end{equation}
in which $x \in A+f(A),\, y \in f(A),\, z \in A,$ and $t \in A+f(A)$. Note that for any $u,v,w \in A$, a solution of the equation \eqref{equa1} is given by $x = u + f(v) \in A + f(A),\, y = f(v) \in f(A), z = w \in A,$ and $t = w + f(u) \in A+f(A)$. Therefore, let $N$ be the number of solutions of the equation \eqref{equa1}, we have $N \ge |A|^3.$ On the other hand, by using Cauchy-Schwarz inequality, on has
\begin{align*} 
N^2 &\leq \left|f(A) + A \right| \cdot \left|\left\{ (x,y,z,z',y',z') \in \left((A+f(A)) \times f(A)\times A\right)^{2}: f(x-y)+z=f(x'-y')+z'\right\}\right| \notag \\
 &=: \left| f(A) + A \right| \cdot |E|
\end{align*}
Hence, 
\begin{equation}\label{iequa2}
|A|^6 \leq \left| A + f(A) \right| \cdot |E|
\end{equation}
Similarly, to bound $|E|$, we apply Theorem \ref{P-P-Y} for the following set of points and set of planes:
\begin{align*}
Q &:= \left\{\left(ax,\, y',\, ax^2+bx - z - ay'^2 + by' \right) \, : \, (x,y',z) \in (A+f(A))\times f(A) \times A \right\} \\
\Pi &:= \left\{ -2y \cdot X + 2ax'\cdot Y + Z = ax'^2+bx' + z'- ay^{2} + by \, : \, (x',y,z') \in (A+f(A)) \times f(A) \times A \right\}.
\end{align*}
Note that $|Q|=|\Pi| \leq |A+f(A)||A||f(A)| \leq |A+f(A)||A|^{2}$ and $|E| \leq I(Q,\Pi)$, where $I(Q,\Pi)$ is number of incidences between $Q$ and $\Pi$. Therefore, it follows from \eqref{iequa2} and applying Theorem \ref{P-P-Y}, one gets
\begin{align*}
|A|^{6} &\leq |A+f(A)|I(Q,\Pi) \\ 
&\leq |A+f(A)|\left(\dfrac{|A+f(A)|^{2}|A|^{4}}{q^{r}}+ q^{2r-1}|A+f(A)||A|^{2} \right) 
\end{align*}
Hence,
\[ |A|^{2}q^{r} \leq |A+f(A)|^{3} \left(1+\frac{q^{3r-1}}{|A+f(A)||A|^{2}}\right).\]
Therefore, if $|A+f(A)||A|^{2 \geq q^{3r-1}}$ then we have  
\begin{align*}
|f(A)+A| \geq 2^{-\frac{1}{3}}|A|^{2/3}q^{r/3},
\end{align*}
which completes the proof of Theorem \ref{theo.f(A)+A}.

\section{Proof of Theorem \ref{AA+AA}}
For each $m \in \mathcal{R},$ consider the equation 
\begin{equation*}
xy + zt = m , \quad \text{where} \quad \, x,y,z,t \in A.
\end{equation*}
It follows from Cauchy-Schwarz inequality that
\begin{align*}
|A|^{8} &\leq |AA+AA| \left| \left\{ (x,y,z,t,x',y',z',t')\in A^8: xy+zt=x'y'+z't' \right\} \right| \\
&\leq |AA+AA| \cdot |A-A|^2 \left|\left\{ (x,y,t,x',t',y',\alpha,\beta):\,  xy+xt+\alpha t = x'y'+x't'+\beta t'\right\}\right|,
\end{align*}
where $(\alpha ,\beta) \in (A-A) \times (A-A)$ is any chosen pair, and $(x,y,t,x',y',t') \in A^6.$

Define the point set $Q$ and plane set $\Pi$ as follows
\begin{align*}
Q &:= \left\{ (x,y',xt+\alpha t): (x,y',t) \in A^3 \right\},\\
\Pi &:= \left\{ y \cdot X - x' \cdot Y + Z = x't'+\beta t': \, (y,x',t') \in A^3 \right\}.
\end{align*}
Note that $|Q|,|\Pi|\leq |A|^{3}$, and the number of incidences between $Q$ and $\Pi$ satisfies
\[ \left|\{ (x,y,t,x',t',y',\alpha,\beta): xy+xt+\alpha t=x'y'+x't'+\beta t'\}\right| \leq I(Q,\Pi). \]
Then, applying Theorem \ref{P-P-Y}, we obtain 
\begin{align*}
|A|^{8} &\leq |AA+AA||A-A|^2 \cdot I(Q,\Pi) \\
&\leq |AA+AA||A-A|^2 \cdot \left(\frac{|A|^{6}}{q^{r}}+ q^{2r-1}|A|^{3} \right)
\end{align*}
Therefore, if $|A|^{3} \geq q^{3r-1}$ then 
\[ \dfrac{1}{2}|A|^{2}q^{r} \leq |AA+AA||A-A|^{2},\]
or equivalently,
\begin{equation*}
\max\left\{ |A-A|, |AA+AA| \right\} \geq 2^{-\frac{1}{3}}|A|^{2/3}q^{r/3}.
\end{equation*}
This completes the proof of the theorem.

\section{Proof of Theorem \ref{AA.A^d+A^d}}

We first observe that if $|A| \ll q^{r - 2/5}$ then 
\[ |A^d + A^d||AA|^2 \gg |A||AA|^2 \gg \dfrac{q^{6r-2}}{|A|^3} \gg q^r|A|^2.\] 
Therefore, we only need to consider the case of $|A| \gg q^{r-2/5}.$ Since $|\mathcal{R}^0| = q^{r-1},$ without loss of generality, we can assume that $A \subset \mathcal{R}^*.$ We define the $d$-power energy of $A$ by
\[  E_{d}(A)= \left|\left\{ a^{d}+b^{d}\,=\,c^{d}+e^{d}: a,b,c,d \in A \right\} \right|.\] 
For $f,g \in A$, we have
\begin{align*}
E_{d}(A) &= |A|^{-2} \left|\left\{ a^{d} + \frac{(bf)^{d}}{f^{d}} = c^{d} + \dfrac{(eg)^{d}}{g^{d}}\right\}\right| \\
&\leq |A|^{-2}\left|\left\{ a^{d}+\frac{h^{d}}{f^{d}} = c^{d}+\dfrac{k^{d}}{g^{d}}: a,f,c,g \in A,\,  h,k \in AA \right\}\right|.
\end{align*}
We define the multi-set of points with coordinates $\left(a^{d},\, 1/f^{d}, \,k^{d}\right)$ and the multi-set of planes with equation $X+ h^{d} \cdot Y-\left(\dfrac{1}{g^{d}}\right) \cdot Z = c^{d}$. Since $A \subset \mathcal{R}^*$, the number of solutions of the equation $x^d=t$ for $x,t \in A$ is at most $d$. Therefore, these points and planes may have weights up to $d^3$, implying that $\omega \leq d^3,$ and $W=2|A|^{2}|AA|.$

From the point-plane weighted incidences (Theorem \ref{CorY1}) and using the Cauchy Schwarz inequality, we have
\begin{align*}
|A|^{4} \ll E_{d}(A) |A^{d}+A^{d}|.
\end{align*}
Hence,
\[ |A|^{6} \ll \left| A^{d}+A^{d}\right|\left( \frac{1}{q^{r}}|A|^{4}|AA|^{2}+q^{2r-1}|A|^{2}|AA| \right). \]
In other words,
\begin{align*} 
q^{r}|A|^{2} \ll \left|A^{d}+A^{d}\right| |AA|^{2} \left( 1+\frac{q^{3r-1}}{|AA||A|^{2}} \right).
\end{align*}
Therefore, if $|AA|.|A|^{2} \geq q^{3r-1}$ then
\begin{align*}
\left|A^{d}+A^{d}\right| |AA|^{2} \gg q^{r}|A|^{2},
\end{align*}
which completes the proof of Theorem \ref{AA.A^d+A^d}.

\section{Further Remarks}

\begin{itemize}

\item Note that, using the same arguments as in the proof of Theorem \ref{mainth3-1}, we cannot extend the theorem for more general quadratic polynomials $f(x,y,z) = axy + bxz + R(s) + S(y) + T(z)$ and $f(x,y,z) = axy + bxz + czy + R(s) + S(y) + T(z)$. The main reason is that we cannot control the number of solutions of the equations $T(z) + bxz = n$ and $T(z) + bxz + cyz = n$ as these equations may have up to $q^{r-1}$ solutions.

\item Given the commutative ring $\mathcal{R}$ and its units $\mathcal{R}^*$, we can define its ring of fraction $\mathcal{Q}= (\mathcal{R}^*)^{-1}\mathcal{R}.$ This ring of fraction defines an equivalence relation on $\mathcal{R} \times \mathcal{R}^*$ that $(m,m') \sim (r,r')$ if and only if $mr'=m'r$. We write $\dfrac{m}{m'}=\dfrac{r}{r'}$ if $(m,m') \sim (r,r')$. 

We can define the collinearity property in $\mathcal{R} \times 
\mathcal{R}$: three points $(a,a'),(b,b'),(c,c') \in \mathcal{R}\times \mathcal{R}$ are collinear if there exists a number $k \in \mathcal{Q}$ such that $(a-b)=k(c-b)$ and $(a'-b')=k(c'-b')$. The set of points, which are collinear to an initial point $(u,v)$ via a fixed number $k \in \mathcal{Q}$, is a line.

Using the same technique, we have the following geometric type result.

\begin{theorem}\label{T(A).L(P)}
Let $ T(P) $ be the number of collinear triples in the point set $ P = A \times A \subset \mathcal{R} \times \mathcal{R}$ and $L(P)$ be the set of lines connecting pairs of distinct points of $P= A \times A$, then
 \begin{align*} 
 T(P)\leq q^{2r-1}\,|A|^{3}+ \frac{|A|^{6}}{q^{r}}\, + 2.|A|^{4},
 \end{align*}
 and
 \begin{align*} 
 |L(P)| \gg \min\left\{ q^{2r}, \frac{|A|^{6}}{q^{4r-2}}\right\}.
 \end{align*}
\end{theorem}
\begin{proof}
Given that three points $(a,a'),(b,b'),(c,c') \in P$ are collinear. We consider the following cases.
\begin{itemize}
\item[1.] If $a=b$, then  the collinearity implies $c=b$. The number of collinear triples in this case is $|A|^{4}$. 
\item[2.] If $a'=b'$, then similarly  $c'=b'$. Thus, the number of collinear triples in this case is also $|A|^{4}$.
\item[3.] If $a \neq b, c \neq b; a' \neq b', c' \neq b'$, then the collinearity is expressed by the condition
$(a-b)(c'-b')=(a'-b')(c-b)$.

\end{itemize}
Thus, all cases together imply:
\begin{eqnarray} 
T(P) &\leq & \left|\left\{ (a,b,c,a',b',c') \in A^{6}: (a-b)(c'-b')=(a'-b')(c-b) : \, a \neq b \neq c ,a' \neq b' \neq c' \right\}\right| \nonumber \\ & & + 2|A|^{4}.\label{TA}
\end{eqnarray}
Define a point set $Q$ and plane set $\Pi$ as follows
\begin{align*}
Q &:= \left\{ \left(c-b,\,b,\,a'(c-b)\right):\, b\neq c,\, (a',b,c) \in A \times A \times A \right\},\\
\Pi &:=\left\{ b' X - Y(c'-b') - Z = -a(c'-b'): \, b' \neq c', (a,b',c') \in A \times A \times A \right\}.
\end{align*}
Note that $|Q|=|\Pi| \leq |A|^{3}$, and the number of incidences $I(Q, \Pi)$ between $Q$ and $\Pi$ is
\[ \left|\left\{ (a,b,c,a',b',c') \in A^{6}: \, (a-b)(c'-b')=(a'-b')(c-b), \, a\neq b \neq c , \, a' \neq b' \neq c' \right\}\right|.\]
Therefore, applying Theorem \ref{P-P-Y}, we have
\begin{align*}
|I(Q,\Pi)| &\leq \dfrac{1}{q^{r}}\,|Q|\,|\Pi|+ q^{2r-1}\,|Q|^{1/2}\,|\Pi|^{1/2}\\
&\leq  \frac{1}{q^r}\,|A|^{6}+q^{2r-1}|A|^{3}.
\end{align*}
Together with (\ref{TA}), we have
\begin{align*}
|T(A)|\leq q^{2r-1}\,|A|^{3}+ \frac{|A|^{6}}{q^{r}} + 2|A|^{4},
\end{align*}
which completes the proof of the first part of Theorem \ref{T(A).L(P)}.

Now, we are going to prove the second part. By the definition of $L(P)$ and the H\"{o}lder inequality, we have
\[ |A|^{4} \ll \displaystyle \sum_{l \in L(P)} n(l)^{2}\ll |L(P)|^{\frac{1}{3}}\,T(A)^{\frac{2}{3}},\]
in which $n(l)$ is the number of points $P$ supported on line $l \in L(P)$.\\
Using the result in the first part, we have
\begin{align*}
|A|^{4} \ll |L(P)|^{1/3} \left(q^{2r-1}\,|A|^{3} + 2|A|^{4} + \frac{|A|^{6}}{q^{r}}\right)^{2/3}.
\end{align*}
Hence,
\[ |A|^{6} \ll \left|L(P)\right|^{1/2}\left(\frac{|A|^{6}}{q^{r}} + q^{2r-1}|A|^{3}\right).\]
In other words, we conclude that
\[ L(P)^{1/2} \gg \dfrac{q^{r}\,|A|^{6}}{|A|^{6}+q^{3r-1}.\,|A|^{3}}=\frac{q^{r}\,|A|^{3}}{|A|^{3}+q^{3r-1}} \gg \min\left\{ q^{r},\, \frac{|A|^3}{q^{2r-1}}\right\},\]
which completes the proof of second part of Theorem \ref{T(A).L(P)}. 
\end{proof}
\end{itemize}

\bibliographystyle{amsplain}

\end{document}